\documentclass[envcountsect]{llncs}
\usepackage{enumerate}
\usepackage{amsfonts, amsmath, amssymb}
\usepackage{graphicx}

\newcommand{\defin}[1]{\emph{\textbf{#1}}}

\newcommand{\N}{\mathbb{N}}

\renewcommand{\P}{\mathcal{P}}
\newcommand{\Q}{\mathbb{Q}}
\newcommand{\R}{\mathbb{R}}
\renewcommand{\S}{\mathcal{S}}

\newcommand{\dom}{\text{dom}}

\newcommand{\Baire}{\mathcal{N}}

\newcommand{\res}[1]{\mathop{\upharpoonright_{#1}}}

\newcommand{\Po}{\P(\omega)}
\newcommand{\fc}{\sqsubseteq}
\newcommand{\fs}{\N^*}

\author{Mathieu Hoyrup\inst{1}, Crist\'obal Rojas\thanks{This author was supported by Marie Curie RISE project CID}\inst{2},
Victor  Selivanov\thanks{This author was supported by Inria program Invited Researcher, and the Regional Mathematical Center of Kazan Federal University (project 1.13556.2019/13.1 of the Ministry of Education and Science of Russian Federation)}\inst{3}, Donald M. Stull\thanks{This author was supported by Inria post-doc program}\inst{1}}

\authorrunning{M. Hoyrup, C. Rojas, V. Selivanov, D. Stull}

\institute{Universit\'e de Lorraine, CNRS, Inria, LORIA, F 54000 Nancy, France\\
{\tt mathieu.hoyrup@inria.fr, donald.stull@inria.fr}
\and
Universidad Andres Bello, Santiago, Chile\\
{\tt crojas@mat-unab.cl}
\and
A.P. Ershov
Institute of Informatics
Systems, Novosibirsk\\
and Kazan Federal University, Russia\\
   {\tt vseliv@iis.nsk.su}}

\title{Computability on quasi-Polish spaces\thanks{This project has received funding from the European Union’s Horizon 2020 research and innovation programme under the Marie Skłodowska-Curie grant agreement No 731143 \protect \includegraphics[height=3mm]{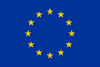}}}

\begin{document}
\maketitle

\begin{abstract}
We investigate the effectivizations of several equivalent definitions of quasi-Polish spaces and study which characterizations hold effectively. Being a computable effectively open image of the Baire space is a robust notion that admits several characterizations. We show that some natural effectivizations of quasi-metric spaces are strictly stronger.
\end{abstract}

\section{Introduction}\label{in}

Classical  descriptive set theory (DST) \cite{ke95} deals with
hierarchies of sets, functions and equivalence relations in Polish
spaces.  Theoretical Computer Science, in particular Computable
Analysis \cite{wei00}, motivated an extension of the classical DST
to non-Hausdorff spaces; a noticeable progress was achieved for the
$\omega$-continuous domains and quasi-Polish spaces \cite{s06,br}. The theory of quasi-Polish spaces is already a well-established part of classical DST \cite{br,chen}

Theoretical  Computer Science and Computable Analysis   especially
need an effective  DST for some effective versions of the mentioned
classes of topological spaces. A lot of useful work in this
direction was done in Computability Theory but only
for the discrete space $\N$, the Baire space $\mathcal{N}$, and
some of their  relatives \cite{ro67,sg09}. For a systematic work to
develop the effective DST for effective Polish spaces see e.g.
\cite{mo09,lo78,gr15}. There was also some  work on the effective
DST for effective domains and approximation spaces
\cite{s06,s08,bg12,s15}.

In this paper we continue the search of a ``correct'' version of a computable quasi-Polish space initiated in \cite{s15,kk17}. By a correct version we mean one having properties similar to effective versions of those in the classical case: the computable quasi-Polish spaces have to subsume the well established classes of computable Polish spaces and computable $\omega$-continuous domains and to admit a good enough effective DST.

We identify effective versions of quasi-Polish spaces satisfying these specifications. One of them is the class of computable effectively open images of the Baire space identified and studied in \cite{s15}. We provide some characterizations of this class which are effective versions of the corresponding  known characterizations of quasi-Polish spaces in \cite{br}. However we show that some natural effectivizations of complete quasi-metric spaces are strictly stronger.

The results of this paper were obtained in September 2018 during a research stay of the second and third authors in Inria, Nancy. On the final stage of preparing this paper the preprint \cite{br1} appeared where some of our results where obtained independently (using a slightly different approach and terminology), notably Definition \ref{def} and Theorem \ref{thm_char}. 

In order to make our discussion of effective spaces closer to the corresponding classical theory, we use an approach based on the canonical embeddings of cb$_0$-spaces into the Scott domain $\Po$ and on the computability in this domain. This approach (which emphasizes the notion of effective continuity rather than the equivalent notion of computability w.r.t. admissible representations more popular in Computable Analysis) was promoted in \cite{kk08,s08}.

We start  in the next section with recalling definitions of some notions  of effective spaces and of effective DST in such spaces; we also try to simplify and unify rather chaotic terminology in this field. In Sections \ref{prel} and \ref{sec_pi2} we establish the main technical tools used in the sequel. In Section \ref{sec_qp}  we propose a definition of effective quasi-Polish spaces and prove characterizations of this notion. In Section \ref{sec_qm} we propose two effective notions of quasi-metric space and prove that they differ from the notion of effective quasi-Polish space. 

\section{Preliminaries}\label{prel}

Here we recall some known notions and facts, with a couple of new observations. 

Notions similar to those considered below were
introduced (sometimes independently) and studied in \cite{gsw07}, \cite{kk08}, \cite{s08}
under different names. We  use a slightly different terminology, trying to simplify it and make it closer to that of classical topology. Note that the terminology in   effective topology  is still far from being fixed.
All topological spaces considered in this paper are assumed to be countably based. Such a space satisfying the $T_0$-separation axiom is sometimes called a cb$_0$-space, for short. We recall that~$(W_e)_{e\in\N}$ is some effective enumeration of the computably enumerable (c.e.) subsets of~$\N$.

\begin{definition}
An \defin{effective topological space} is a countably-based $T_0$ topological space coming with a numbering~$(B^X_i)_{i\in\N}$ of a basis, such that there is a computable function~$f:\N^2\to\N$ such that~$B^X_i \cap B^X_j=\bigcup_{k\in W_{f(i,j)}}B^X_k$.
\end{definition}

Many popular spaces (e.g., the discrete space $\N$ of naturals, the space of reals ${\mathbb R}$, the domain $\Po$, the Baire space $\mathcal{N}=\N^\N$, the Cantor space  $\mathcal{C}=2^\N$ and the Baire domain $\N^{\leq\N}=\N^*\cup\N^\N$ of finite and infinite sequences of naturals) are effective topological spaces in an obvious way. The effective topological space $\N$ is trivial  topologically but very interesting for Computability Theory. We use some almost standard notation related to the Baire space. In particular,  $[\sigma]$ denotes the basic open set induced by $\sigma\in\N^*$ consisting of all $p\in\mathcal{N}$ having $\sigma$ as an initial segment; we sometimes call such sets cylinders. Let $\varepsilon$ denote the empty string in $\N^*$.  

In \cite{s82,s08} the effective Borel and effective Hausdorff hierarchies in arbitrary effective topological spaces $X$ were introduced. Also the effective Luzin hierarchy is defined naturally \cite{s15}. Below we use the simplified notation for levels of these hierarchies like $\Pi^0_n(X)$, $\Sigma^1_n(X)$ or $\Sigma^{-1}_n(X)$ (which naturally generalizes the notation in computability theory) and some of their obvious properties. We will also use the expression \emph{effective open set} for sets in the class~$\Sigma^0_1(X)$, which are the sets~$\bigcup_{i\in W}B^X_i$ for some c.e.~set~$W\subseteq\N$.
\begin{definition}
If~$X,Y$ are effective topological spaces then a function~$f:X\to Y$ is \defin{computable} if the sets~$f^{-1}(B^Y_i)$ are uniformly effective open sets.
\end{definition}
As observed in \cite{s15}, for any  effective topological space  $X$, the equality
relation~$=_X$ on $X$ is in $\Pi^0_2(X\times X)$. The argument in \cite{s15} shows that also the specialization partial order $\leq_X$ has the same descriptive complexity. In particular, every singleton is in the boldface class~$\boldsymbol{\Pi}^0_2(X)$.

With any effective topological space $X$ we associate the {\em canonical embedding}~$e:X\to \Po$ defined by $e(x)=\{n\mid x\in B^X_n\}$ (in \cite{s08} the canonical embedding was denoted as $O_\xi$; we changed the notation here to make it closer to that of the paper \cite{chen} which is cited below). The canonical embedding is a computable homeomorphism between $X$ and the subspace $e[X]$ of $\Po$. It can be used to study computability on cb$_0$-spaces \cite{kk08,s08} using the fact that the computable functions on $\Po$ coincide with the enumeration operators \cite{ro67}.

The more popular and general approach to computability on topological spaces is based on representations \cite{wei00}. The relation between the two approaches is based on the so called enumeration representation~$\rho:\mathcal{N}\to \Po$ defined by $\rho(x)=\{n\mid\exists i(x(i)=n+1)\}$. The function $\rho$ is a computable effectively open surjection. The canonical embedding $e$ induces the \defin{standard representation}~$\rho_X=e^{-1}\circ\rho_A$ of $X$ where $A=e(X)$ and $\rho_A$ is the restriction of $\rho$ to~$\rho^{-1}(A)$. The function $\rho_X$ is a computable effectively open surjection. We will implicitly identify any effective topological space~$X$ with its image under the canonical embedding, so that~$X$ is a subspace of~$\Po$, and~$\rho_X$ is the restriction of~$\rho$ to~$\rho^{-1}(X)$.

Note that for effective topological spaces~$X$ and~$Y$, $f:X\to Y$ is computable iff there exists a computable function~$F:\dom(\rho_X)\to\dom(\rho_Y)$ such that~$\rho_Y\circ F=f\circ \rho_X$.

\section{Results on~$\Pi^0_2$-sets}\label{sec_pi2}
This section contains the technical tools that will be used to prove the characterizations of effective quasi-Polish spaces.

\begin{definition}
Let~$X$ be an effective topological space. We say that~$A\subseteq X$ is \defin{computably overt} if the set~$\{i\in\N:B^X_i\cap A\neq\emptyset\}$ is c.e.
\end{definition}

Observe that the overt information does not uniquely determine the set, but only its closure. In the literature, overt and computably overt sets are often assumed to be closed. It is important to note that in this paper, no such assumption is made.
 
We recall that if~$X$ is an effective topological space then a set is in~$\Pi^0_2(X)$ if it is an intersection of Boolean combinations of uniformly effective open subsets of~$X$. We prove an effective version of Theorem 68 in \cite{br}.

\begin{lemma}\label{lem_pi2}
Let~$X$ be an effective topological space. For~$A\subseteq X$,
\begin{itemize}
\item $A\in \Pi^0_2(X)$ iff~$\rho_X^{-1}(A)\in\Pi^0_2(\dom(\rho_X))$,
\item $A$ is computably overt iff $\rho_X^{-1}(A)$ is computably overt.
\end{itemize}
\end{lemma}
\begin{proof}
If~$A\in\Pi^0_2(X)$ then one easily obtains~$\rho_X^{-1}(A)\in\Pi^0_2(\dom(\rho_X))$. We now prove that if~$\rho_X^{-1}(A)\in \Sigma^0_2(\dom(\rho_X))$ then~$A\in\Sigma^0_2(X)$, which implies the same result for the class~$\Pi^0_2$. Let~$\rho_X^{-1}(A)=\dom(\rho_X)\cap\bigcup_n U_n\setminus V_n$ where~$U_n,V_n\in\Sigma^0_1(\Baire)$ uniformly and~$V_n\subseteq U_n$. Then the set
\begin{equation*}
A':=\bigcup_{\sigma,n}\rho_X([\sigma]\cap U_n)\setminus \rho_X([\sigma]\cap V_n)
\end{equation*}
belongs to~$\Sigma^0_2(X)$, because the image of a~$\Sigma^0_1(\Baire)$-set is a~$\Sigma^0_1(X)$-set uniformly. We show that~$A=A'$. The inclusion~$A'\subseteq A$ is straightforward. For the other inclusion, let~$x\in A$. One has~$\rho_X^{-1}(x)\in\boldsymbol{\Pi}^0_2(\Baire)$ so~$\rho_X^{-1}(x)$ is quasi-Polish so it is a Baire space (\cite{br}). One has~$\rho_X^{-1}(x)\subseteq\bigcup_n U_n\setminus V_n$. By Baire category, there exists~$n$ such that~$\rho_X^{-1}(x)\cap U_n\setminus V_n$ is somewhere dense in~$\rho_X^{-1}(x)$, i.e.~there exists~$\sigma\in\N^*$ such that~$\emptyset\neq \rho_X^{-1}(x)\cap [\sigma]\subseteq U_n\setminus V_n$. As a result,~$x\in \rho_X([\sigma]\cap U_n)\setminus \rho_X([\sigma]\cap V_n)\subseteq A'$.

If~$A$ is computably overt then~$[\sigma]\cap \rho_X^{-1}(A)\neq\emptyset$ iff~$\rho_X([\sigma])\cap A\neq\emptyset$ which is c.e.~as~$\rho_X([\sigma])\in\Sigma^0_1(X)$, uniformly in~$\sigma$.

If~$\rho_X^{-1}(A)$ is computably overt then~$[\sigma]\cap A\neq\emptyset$ iff~$\rho_X^{-1}([\sigma])\cap \rho_X^{-1}(A)\neq\emptyset$ is a c.e.~relation as~$\rho_X^{-1}([\sigma])\in\Sigma^0_1(X)$, uniformly in~$i$.
\end{proof}

An important property of computably overt~$\Pi^0_2$-sets is that they contain computable points. It is a crucial ingredient in the next results.
\begin{proposition}[\cite{ho17}]\label{prop_comp_pi2}
In a computable Polish space, a~$\Pi^0_2$-set is computably overt if and only if it contains a dense computable sequence.
\end{proposition}

Moreover, the next result shows that in a computably overt~$\Pi^0_2$-set, not only can one find an effective indexing over~$\N$ of a dense set of elements, but one can even find an effective indexing over~$\Baire$ of \emph{all} its elements.

\begin{lemma}\label{lem_pi2_baire}
Let~$A\subseteq\Baire$ be non-empty. The following are equivalent:
\begin{enumerate}[(i)]
\item $A$ is a computably overt $\Pi^0_2$-set,
\item There exists a computable effectively open surjective map~$f:\Baire\to A$.
\end{enumerate}
\end{lemma}
When we write that~$f:\Baire\to A$ is open, we mean that for each~$\sigma\in\N^*$, there exists an open set~$U_\sigma\subseteq\Baire$ such that~$f([\sigma])=A\cap U_\sigma$. $f$ is effectively open when~$U_\sigma$ is effectively open, uniformly in~$\sigma$.
\begin{proof}
Assume (i). Let~$A=\bigcap_n A_n$ where~$A_n$ are uniformly effective open sets. We can assume w.l.o.g.~that~$A_{n+1}\subseteq A_n$.

One can build a computable sequence~$(u_\sigma)_{\sigma\in\fs}$ such that~$u_\epsilon=\epsilon$ and:
\begin{itemize}
\item If~$\tau$ properly extends~$\sigma$ then~$u_\tau$ properly extends~$u_\sigma$,
\item If~$|\sigma|=n$ then~$[u_\sigma]\subseteq A_n$,
\item $[u_\sigma]$ intersects~$A$,
\item $[u_\sigma]\cap A$ is contained in~$\bigcup_{i\in\N} [u_{\sigma\cdot i}]$.
\end{itemize}

We build this sequence inductively in~$\sigma$. Given~$u_\sigma$ intersecting~$A$ with~$|\sigma|=n$, one can compute a covering of~$[u_\sigma]\cap A_{n+1}$ with cylinders properly extending~$u_\sigma$ and extract the cylinders intersecting~$A$. Let~$(u_{\sigma\cdot i})_{i\in\N}$ be some computable enumeration of them.

We now define~$f$. For each~$p\in\Baire$, the sequence~$u_{p\res{n}}$ converges to some~$q\in\Baire$. We define~$f(p)=q$. One easily checks that the function~$f:\Baire\to A$ is computable, onto and effectively open as~$f([\sigma])=[u_\sigma]\cap A$.

Now assume (ii). The function~$f$ has a computable right-inverse, i.e.~$g:A\to\Baire$ such that~$f\circ g$ is the identity on~$A$. Indeed, given~$p\in A$, one can enumerate all the cylinders intersecting~$f^{-1}(p)$ as~$[\sigma]\cap f^{-1}(p)\neq\emptyset$ iff~$p\in f([\sigma])$ which can be recognized as~$f([\sigma])$ is effectively open. Hence one can progressively build an element of the closed set~$f^{-1}(p)$.

The function~$g$ is a partial computable function from~$\Baire\to\Baire$. Its domain is~$\Pi^0_2$ and contains~$A$. One has~$p\in A\iff p$ belongs to the domain of~$g$ and~$p=f\circ g(p)$. As a result,~$A$ is~$\Pi^0_2$. The image under~$f$ of a dense computable sequence in~$\Baire$ is a dense computable sequence in~$A$, so the set of cylinders intersecting~$A$ is c.e.
\end{proof}

This result can be extended to subsets of~$\Po$.
\begin{lemma}\label{lem_pi2_po}
Let~$A\subseteq\Po$ be non-empty. The following are equivalent:
\begin{enumerate}[(i)]
\item $A$ is a computably overt $\Pi^0_2$-set,
\item There exists a computable effectively open surjective map~$f:\Baire\to A$.
\end{enumerate}
\end{lemma}
\begin{proof}
If~$A$ is a computably overt~$\Pi^0_2$-set then so is~$\rho^{-1}(A)$, so there exists a computable effectively open onto function~$f:\Baire\to\rho^{-1}(A)$. The function~$\rho\circ f$ satisfies the required conditions.

Conversely, assume that~$f:\Baire\to A$ is a computable effectively open surjective function.
\begin{claim}
There exists a computable function~$g:\rho^{-1}(A)\to\Baire$ such that~$f\circ g=\rho$.
\end{claim}
\begin{proof}[of the claim]
Given~$p\in\rho^{-1}(A)$, let~$A_p=\{q\in\Baire:f(q)=\rho(p)\}$. The set~$A_p$ is a computably overt~$\Pi^0_2$-set relative to~$p$. Indeed, it is~$\Pi^0_2$ relative to~$p$ because equality is~$\Pi^0_2$ in~$\Po$. It is computably overt relative to~$p$ because a cylinder~$[\sigma]$ intersects~$A_p$ iff~$\rho(p)\in f([\sigma])$ which is a c.e.~relation in~$p$ as~$f$ is effectively open. As a result, by relativizing Proposition \ref{prop_comp_pi2}, one can compute an element in~$A_p$. Everything is uniform in~$p$, so there is a computable function~$g$ mapping each~$p\in\rho^{-1}(A)$ to an element of~$A_p$, hence~$f\circ g(p)=\rho(p)$.
\end{proof}

Now, one has~$q\in\rho^{-1}(A)$ iff~$g(q)$ is defined and~$f\circ g(q)=\rho(q)$. Both relations are~$\Pi^0_2$, so~$\rho^{-1}(A)\in\Pi^0_2(\Baire)$ hence~$A\in\Pi^0_2(\Po)$ by Lemma \ref{lem_pi2}. Moreover,~$A$ is computably overt because for each basic open set~$B$ of~$\Po$,~$B\cap A\neq\emptyset$ iff~$f^{-1}(B)\neq\emptyset$, which is a c.e.~relation. 
\end{proof}

\section{Effective quasi-Polish spaces}\label{sec_qp}

According to Theorem 23 of \cite{br}, the quasi-Polish spaces (defined originally as the countably based completely quasi-metrizable spaces) coincide with the continuous open images of the Baire space. Effectivizing this definition, we obtain the following candidate for a notion of effective quasi-Polish space.

\begin{definition}\label{def}
An effective topological space~$X$ is an \defin{effective quasi-Polish space} if~$X$ is the image of~$\Baire$ under a computable effectively open map, or~$X$ is empty.
\end{definition}

Of course this notion is preserved by computable homeomorphisms (bijections that are computable in both directions).
\begin{theorem}\label{thm_char}
Let~$X$ be an effective topological space with its standard representation~$\rho_X$. The following statements are equivalent:
\begin{enumerate}
\item $X$ is effective quasi-Polish,
\item The image of~$X$ under its canonical embedding in~$\Po$ is a computably overt~$\Pi^0_2$-subset of~$\Po$,
\item $\dom(\rho_X)$ is a computably overt~$\Pi^0_2$-subset of~$\Baire$.
\end{enumerate}
\end{theorem}
\begin{proof}
The equivalence~$1. \iff 2.$ is the content of Lemma \ref{lem_pi2_po}. The equivalence~$2. \iff 3.$ is the content of Lemma \ref{lem_pi2} for the space~$\Po$.
\end{proof}

We also formulate the effective version of Theorem 21 of \cite{br}.

\begin{theorem}\label{thm_subspace}
Let~$X$ be an effective quasi-Polish space. A subspace~$Y\subseteq X$ is an effective quasi-Polish space iff~$Y$ is a computably overt~$\Pi^0_2$-subset of~$X$.
\end{theorem}
\begin{proof}
Via the canonical embedding, we have~$Y\subseteq X\subseteq\Po$. We start by assuming that~$Y$ is a computably overt~$\Pi^0_2$-subset of~$X$. As~$Y\in\Pi^0_2(X)$ and~$X\in\Pi^0_2(\Po)$, one has~$Y\in\Pi^0_2(\Po)$. To show that~$Y$ is computably overt in~$\Po$, simply observe that for a basic open set~$B$ of~$\Po$,~$B$ intersects~$Y$ iff~$B^X:=B\cap X$ intersects~$Y$. It is a c.e.~relation~as~$Y$ is computably overt in~$X$.

If~$Y$ is effective quasi-Polish then it is a computably overt~$\Pi^0_2$-subset of~$\Po$ so it is a computably overt~$\Pi^0_2$-subset of~$X$, which is a subspace of~$\Po$.
\end{proof}

Sometimes it is easier to work with a computably admissible representation other than the standard representation.

\begin{theorem}\label{thm_adm_rep}
Let~$X$ be an effective topological space. If~$X$ admits a computably admissible representation whose domain is a computably overt~$\Pi^0_2$-subset of~$\Baire$, then~$X$ is an effective Polish space.
\end{theorem}
\begin{proof}
Let~$\delta$ be a computably admissible representation of~$X$ such that~$\dom(\delta)$ is a computably overt~$\Pi^0_2$-subset of~$\Baire$. By definition of computably admissible,~$\delta$ is computably equivalent to~$\rho_X$, i.e.~there exist partial computable functions~$F,G:\subseteq\Baire\to\Baire$ satisfying~$\rho_X=\delta\circ F$ and~$\delta=\rho_X\circ G$. We show that~$\dom(\rho_X)$ is a computably overt~$\Pi^0_2$-set.

We recall that~$\rho_X$ is the restriction of the representation~$\rho$ of~$\Po$ to~$\rho^{-1}(X)$. We show that~$\dom(\rho_X)=\rho^{-1}(X)$ is a computably overt~$\Pi^0_2$-set. Let~$p\in\Baire$. One has~$p\in\dom(\rho_X)=\rho^{-1}(X)$ iff~$F(p)$ is defined,~$F(p)\in\dom(\delta)$ and~$\rho(G\circ F(p))=\rho(p)$. All these conditions are~$\Pi^0_2$, so~$\rho^{-1}(X)$ belongs to~$\Pi^0_2(\Baire)$.

One has~$[\sigma]\cap\dom(\rho_X)\neq\emptyset$ iff~$\delta^{-1}(\rho([\sigma]))\neq\emptyset$ which is c.e.~in~$\sigma$ as~$\dom(\delta)$ is computably overt.
\end{proof}

\section{Effective quasi-metric spaces}\label{sec_qm}
We now propose two effective versions of quasi-metric spaces and compare them with the notion of effective quasi-Polish space. A quasi-metric on a set~$X$ is a function~$d:X\times X\to \R_{\geq 0}$ satisfying:
\begin{itemize}
\item $d(x,z)\leq d(x,y)+d(y,z)$,
\item $x=y$ iff~$d(x,y)=d(y,x)=0$.
\end{itemize}

The quasi-metric~$d$ induces a metric~$\hat{d}(x,y)=\max(d(x,y),d(y,x))$.
\begin{definition}
A \defin{computable quasi-metric space} is a triple~$(X,d,S)$ where $d$ is a quasi-metric on~$X$ and~$S=\{s_i\}_{i\in\N}$ is a~$\hat{d}$-dense sequence such that~$d(s_i,s_j)$ are uniformly computable.
\end{definition}

We recall that a real number~$x$ is \defin{right-c.e.}~if~$x=\inf_i q_i$ for some computable sequence of rationals~$(q_i)_{i\in\N}$.
\begin{definition}
A \defin{right-c.e.~quasi-metric space} is a triple~$(X,d,S)$ where~$d$ is a quasi-metric on~$X$ and~$S=\{s_i\}_{i\in\N}$ is a~$\hat{d}$-dense sequence such that~$d(s_i,s_j)$ are uniformly right-c.e.
\end{definition}


Every right-c.e.~quasi-metric space is an effective topological space with the basis of balls~$B(s,r)=\{x\in X:d(s,x)<r\}$ with~$s\in S$ and~$r$ positive rational. To see this, we consider formal inclusion between balls:~$B(s,q)\fc B(t,r)$ iff~$d(t,s)+q<r$. Formal inclusion is c.e.~and~$B_i\cap B_j=\bigcup_{k:B_k\fc B_i\text{ and }B_k\fc B_j}B_k$, so the axiom of effective topological spaces is satisfied. As a result, any such space has its standard representation~$\delta_S$. We define another representation.

\begin{definition}
The \defin{Cauchy representation}~$\delta_C$ is defined in the following way: a point~$x\in X$ is represented by any sequence~$s_n\in S$ such that~$d(s_n,s_{n+1})<2^{-n}$ and~$s_n$ converges to~$x$ in the metric~$\hat{d}$.
\end{definition}

\begin{theorem}
On a right-c.e.~quasi-metric space, the Cauchy representation is computably equivalent to the standard representation.
\end{theorem}

\begin{proof}
For the proof we will also consider a slightly different representation~$\delta'_C$ where~$x$ is represented by any sequence~$s_n$ such that~$d(s_n,x)<2^{-n}$ and~$s_n$ converges to~$x$ in~$\hat{d}$.

We prove the following computable reductions:~$\delta_C\leq \delta'_C\leq \delta_S\leq\delta_C$.

Proof of~$\delta_C\leq \delta'_C$. Assume we are given a~$\delta_C$-name of~$x$, which is essentially a sequence~$(s_n)_{n\in\N}$ such that~$d(s_n,s_{n+1})<2^{-n}$ and~$\lim_{n\to\infty}\hat{d}(s_n,x)=0$. One easily checks that the sequence~$(s_{n+1})_{n\in\N}$ is a~$\delta_C$-name for~$x$.

Proof of~$\delta'_C\leq\delta_S$. Assume we are given a~$\delta'_C$-name of~$x$, which is essentially a sequence~$s_n$ such that~$d(s_n,x)<2^{-n}$ and~$\lim_{n\to\infty}\hat{d}(s_n,x)=0$. We show that we can enumerate the basic balls containing~$x$. Indeed, we show that~$x\in B(s,r)$ if and only if there exists~$n$ such that~$d(s,s_n)<r-2^{-n}$. First assume that the latter inequality holds. By the triangle inequality,
\begin{equation*}
d(s,x)\leq d(s,s_n)+d(s_n,x)<r-2^{-n}+2^{-n}=r.
\end{equation*}
Conversely, if~$d(s,x)<r$ then as~$\hat{d}(s_n,x)$ converges to~$0$, for sufficiently large~$n$ one has~$d(x,s_n)+2^{-n}<r-d(s,x)$ so~$d(s,s_n)\leq d(s,x)+d(x,s_n)<r-2^{-n}$.

Proof of~$\delta_S\leq\delta_C$. Assume we are given an enumeration of the basic balls containing~$x$, call it~$U_1,U_2,U_3,\ldots$. We build a sequence~$(s_n)_{n\in\N}$ as follows.

We take~$s_0$ such that~$d(s_0,x)<1$, which we can find by looking for a ball of radius~$1$ containing~$x$. Once~$s_0,\ldots,s_n$ have been defined, we look for~$s_{n+1}$ satisfying:
\begin{itemize}
\item $s_{n+1}\in U_1\cap\ldots\cap U_{n+1}$,
\item $d(s_n,s_{n+1})<2^{-n}$,
\item $d(s_{n+1},x)<2^{-n-1}$.
\end{itemize}

Such a point must exist, as if~$x'$ is sufficiently~$\hat{d}$-close to~$x$, the first and third conditions are satisfied, and~$d(s_n,x')\leq d(s_n,x)+d(x,x')<2^{-n}+d(x,x')$ by induction hypothesis, so~$d(s_n,x')<2^{-n}$ if~$d(x,x')$ is sufficiently small. Such a point can be effectively found,~$d$ is right-c.e.~on~$S$.

The sequence~$(s_n)_{n\in\N}$ satisfies the conditions of being a~$\delta_C$-name of~$x$. Indeed,~$d(x,s_n)$ converge to~$0$, as for each rational~$\epsilon$ there exists~$s\in S$ such that~$\hat{d}(s,x)<\epsilon$, so the ball~$B(s,\epsilon)$ appears as some~$U_i$, so for~$n\geq i$,~$d(x,s_n)\leq d(x,s)+d(s,s_n)<2\epsilon$.
\end{proof}

We recall that a quasi-metric~$d$ is (Smyth-)complete if every Cauchy sequence converges in the metric~$\hat{d}$, and that a space is quasi-Polish iff it is completely quasi-metrizable \cite{br}. One direction of this equivalence admits an effective version.
\begin{theorem}
Every right-c.e.~quasi-metric space that is complete is an effective quasi-Polish space.
\end{theorem}
\begin{proof}
The domain of the Cauchy representation is a computably overt~$\Pi^0_2$-set. Indeed, the relation~$\forall n,d(s_n,s_{n+1})<2^{-n}$ is~$\Pi^0_2$, and any finite sequence satisfying this condition can be extended (to an ultimately constant sequence, e.g.). As the Cauchy representation is computably equivalent to the standard representation, we can apply Theorem \ref{thm_adm_rep}.
\end{proof}

\begin{proposition}\label{prop_point_rce}
In a right-c.e.~quasi-metric space, the following conditions are equivalent for a point~$x$:
\begin{itemize}
\item $x$ is computable,
\item The numbers~$d(s,x)$ are right-c.e., uniformly in~$s\in\S$.
\end{itemize}
\end{proposition}
\begin{proof}
One has~$x\in B(s,r)\iff d(s,x)<r$. The first relation is c.e.~iff~$x$ is computable. The second relation is c.e.~iff~$d(s,x)$ is right-c.e.
\end{proof}
We recall that a real number~$x$ is left-c.e.~if~$-x$ is right-c.e.
\begin{proposition}\label{prop_comp_point_cp}
In a computable quasi-metric space, the following conditions are uniformly equivalent for a point~$x$:
\begin{itemize}
\item $x$ is computable,
\item The numbers~$d(s,x)$ are right-c.e., uniformly in~$s\in\S$,
\item The numbers~$d(s,x)$ and~$d(x,s)$ are right-c.e.~and left-c.e.~respectively, uniformly in~$s\in\S$.
\end{itemize}
\end{proposition}
\begin{proof}
We only have to prove that for a computable point~$x$, the numbers~$d(x,s)$ are uniformly left-c.e. Let~$(s_n)_{n\in\N}$ be a computable~$\delta_C$-name of~$x$. We show that~$d(x,s)=\sup_n d(s_n,s)-2^{-n}$ which is left-c.e., uniformly in~$s$.

Indeed,~$d(x,s)\geq d(s_n,s)-d(s_n,x)$, and as~$d(x,s_n)\leq \hat{d}(x,s_n)$ converges to~$0$,~$d(x,s)\leq d(x,s_n)+d(s_n,s)$ is arbitrarily close to~$d(s_n,s)$.
\end{proof}

\section{Separation}
Classically, a space is quasi-Polish if and only if it is completely quasi-metrizable \cite{br}. However the proof is not constructive. We know that each right-c.e.~quasi-metric space that is complete is an effective quasi-Polish space, but that the converse fails. For this, we fully characterize the effective notions of quasi-Polish space in a restricted case.

Let~$[0,1]_<$ come with the quasi-metric~$d(x,y)=\max(0,x-y)$, with the rational points as~$\hat{d}$-dense sequence. It is a computable quasi-metric space that is complete. For~$\alpha\in (0,1)$, the subspace~$[\alpha,1]_<$ is an effective topological subspace of~$[0,1]_<$. We investigate when it is an effective quasi-Polish space, a computably quasi-metrizable space, and a right-c.e.~quasi-metrizable space.

\begin{proposition}\label{prop_alpha}
The space~$[\alpha,1]_<$ is an effective quasi-Polish space iff~$\alpha$ is left-c.e.~relative to the halting set.
\end{proposition}
\begin{proof}
Observe that the set~$\{q:(q,1]\cap A\neq\emptyset\}$ is always c.e. Therefore, $[\alpha,1]_<$ is an effective quasi-Polish space if and only if~$[\alpha,1]\in\Pi^0_2([0,1]_<)$ by Theorem \ref{thm_subspace} (the c.e.~conditions is always satisfied as observed above). This is equivalent to the existence of uniformly right-c.e.~numbers~$r_i$ such that~$[\alpha,1]=\bigcap_i (r_i,1]$, i.e.~$\alpha=\sup_i r_i$. This is equivalent to~$\alpha$ being left-c.e.~relative to the halting set.
\end{proof}

\begin{proposition}\label{prop_cpqm}
The space~$[\alpha,1]_<$ admits a computably equivalent computable quasi-metric structure if and only if~$\alpha$ is right-c.e.
\end{proposition}
\begin{proof}
Assume first that~$\alpha$ is right-c.e. There is a computable enumeration~$S=\{q_i\}_{i\in\N}$ of the rational numbers in~$(\alpha,1]$. The quasi-metric~$d(x,y)=\max(0,x-y)$ is computable on~$S$.

Conversely, assume a computable quasi-metric~$d$ with an associate set~$S=\{s_i\}_{i\in\N}$. We now prove that the points~$s_i$ are uniformly computable real numbers, which implies that~$\alpha=\inf_ss_i$ is right-c.e. The function mapping a real number~$x\in[\alpha,1]$ to~$d(x,s_i)$ is left-c.e.~($x$ is given using the standard Cauchy representation). Indeed, from~$x$ one can compute a name of~$x$ in~$[0,1]_<$, from which one can compute a name of~$x$ in~$[\alpha,1]_<$ and we can apply the uniform relative version of Proposition \ref{prop_comp_point_cp}.

The left-c.e.~ function~$x\mapsto d(x,s_i)$ is non-decreasing. Indeed, for~$x\leq x'$, one has~$d(x,s_i)\leq d(x,x')+d(x',s_i)=d(x',s_i)$. Therefore, it can be extended to a left-c.e.~non-decreasing function~$f$ over~$[0,1]$. Indeed, if~$f_0:[0,1]\to\R$ is a left-c.e.~function such that~$f_0(x)=d(x,s_i)$ for~$x\in[\alpha,1]$, then~$f(x):=\inf\{f_0(x'):x'\in[x,1]\}$ is left-c.e.~non-decreasing and agrees with~$d(x,s_i)$ on~$[\alpha,1]$.

As a result, for~$q\in\Q$,~$q>s_i$ if and only if~$f(q)>0$ which is a c.e.~condition, so~$s_i$ is right-c.e. Of course,~$s_i$ is left-c.e.~as it is a computable point of 
$[\alpha,1]_<$.
\end{proof}

\begin{proposition}\label{prop_alpha_ter}
The space~$[\alpha,1]_<$ admits a computably equivalent right-c.e. quasi-metric structure if and only if~$\alpha$ is left-c.e.~or right-c.e.
\end{proposition}
\begin{proof}
If~$\alpha$ is right-c.e.~then there is a computable quasi-metric structure by Proposition \ref{prop_cpqm}. If~$\alpha$ is left-c.e.~then we can take~$S=\{s_i\}_{i\in\N}$ with~$s_i=\max(q_i,\alpha)$, where~$(q_i)_{i\in\N}$ is a computable enumeration of the rational numbers in~$[0,1]$. We can take the restriction of the quasi-metric~$d(x,y)=\max(0,x-y)$. It is right-c.e.~on~$S$. To approximate~$d(s_i,s_j)$ from the right, do the following: if~$q_i\leq q_j$ then output~$0$ (correct as~$s_i\leq s_j$ in that case). If~$q_i>q_j$ then start approximating~$d(q_i,s_j)$ from the right (possible as~$s_j$ is left-c.e.) and switching to~$0$ if we eventually see that~$q_i<\alpha$.

Conversely, assume a right-c.e.~metric structure~$(d',S)$. Given~$q\in\Q\cap[\alpha,1]$,~$d'(s_i,q)$ is uniformly right-c.e. Indeed, each such~$q$ is a computable point of the right-c.e.~quasi-metric space~$([\alpha,1]_<,d',S)$, so by Proposition \ref{prop_point_rce},~$d'(s_i,q)$ is right-c.e.

\begin{claim}
Given~$s\in S$,~$\epsilon>0$, one can compute~$\delta>0$ such that~$B'(s,\delta)\subseteq B(s,\epsilon)$.
\end{claim}
\begin{proof}[of the claim]
The identity from the quasi-metric space~$[\alpha,1]_<$ to the quasi-metric space is computable, so~$B(s,\epsilon)$ is effectively open in~$[\alpha,1]_<$, hence can be expressed as a union of~$d'$-balls. One can find one of them,~$B'(t,r)$, containing~$s$. One has~$d'(t,s)<r$ and~$d'(t,s)$ is right-c.e., so one can compute~$\delta>0$ such that~$d'(t,s)+\delta<r$. One has~$B'(s,\delta)\subseteq B'(t,r)\subseteq B(s,\epsilon)$.
\end{proof}

Let~$\delta_{s,\epsilon}$ be obtained from the previous Claim. Consider thet set~$E=\{q\in\Q\cap[0,1]:\exists s\in S,\epsilon>0, d'(s,q)<\delta_{s,\epsilon}\text{ and }d(s,q)>\epsilon\}$.It is a c.e.~set. It is disjoint from~$[\alpha,1]$: if~$q\in[\alpha,1]$ and~$q\in B'(s,\delta_{s,\epsilon})$ then~$q\in B(s,\epsilon)$. As a result,~$\sup E$ is left-c.e.~and~$\sup E\leq \alpha$. If~$\alpha$ is not left-c.e.~then~$\sup E<\alpha$. As a result, we can fix some rational number~$q_0$ between~$\sup E$ and~$\alpha$, and work with rationals above~$q_0$ only, so that they do not belong to~$E$.

Let~$F=\{(q,\epsilon):q\in\Q\cap [q_0,1],\epsilon>0,\exists s\in S,\text{ such that }d'(s,q)<\delta_{s,\epsilon}\}$. $F$ is c.e.~so~$I:=\inf\{q+\epsilon:(q,\epsilon)\in F\}$ is right-c.e. If~$q>\alpha$ then there must exist~$s\in S$ such that~$s\leq q$, i.e.~$d'(s,q)=0$, so~$(q,\epsilon)\in F$ for every~$\epsilon>0$. As a result,~$I\leq \alpha$. If~$\alpha$ is not right-c.e.~then~$I<\alpha$.

Take~$(q,\epsilon)\in F$ such that~$q+\epsilon<\alpha$. Let~$s\in S$ witness that~$(q,\epsilon)\in F$. One has~$d'(s,q)<\delta_{s,\epsilon}$ and~$d(s,q)\geq d(\alpha,q)>\epsilon$ so~$q\in E$, giving a contradiction.

Therefore,~$\alpha$ is left-c.e.~or right-c.e.
\end{proof}

\begin{corollary}
There exists an effective quasi-Polish space which cannot be presented as a right-c.e.~quasi-metric space.
\end{corollary}
\begin{proof}
Take~$\alpha$ that is left-c.e.~relative to the halting set but neither left-c.e.~nor right-c.e., and apply Propositions \ref{prop_alpha} and \ref{prop_alpha_ter}.
\end{proof}

\section{Discussion and open questions}\label{comqp}

By a {\em computable Polish space} we mean  an effective topological space~$X$ induced  by a computable complete metric space $(X,d,S)$ \cite{wei00,mo09,gkp14}. Most of the popular Polish spaces are computable.

By a {\em computable $\omega$-continuous domain} \cite{aj} we mean a
pair $(X,b)$ where $X$ is an $\omega$-continuous domain and
$b:\N\to X$ is a numbering of a domain base in $X$ modulo which the
approximation relation $\ll$ is c.e. Any computable
$\omega$-continuous domain  $(X,b)$ has the induced effective base  $\beta$ where $\beta_n=\{x\mid b_n\ll
x\}$. Most of the popular  $\omega$-continuous domains are
computable.

By Theorem 1 in \cite{s15}, both the computable Polish spaces and computable $\omega$-continuous domains are computable effective images of the Baire space, hence they are effective quasi-Polish, hence the notion of effective quasi-Polish space introduced in this paper is a reasonable candidate for capturing the computable quasi-Polish spaces. By Theorem 4 in \cite{kk17} (which extends Theorem 4 in \cite{s15}), any effective quasi-Polish space satisfies the effective Suslin-Kleene theorem. By Theorem 5 in \cite{s15}, any effective quasi-Polish space satisfies the effective Hausdorff theorem.

It seems that our search, as well as the independent search in \cite{br1} resulted in natural and convincing candidates for capturing the computable quasi-Polish spaces. Nevertheless, many interesting closely related questions remain open. Since the class of quasi-Polish spaces admits at least ten seemingly different characterizations \cite{br}, the status of effective analogues of these characterizations deserves additional investigation. In particular, this concerns the characterization of quasi-Polish spaces as the subspaces of non-compact elements in ($\omega$-algebraic or $\omega$-continuous) domains.

\bibliographystyle{splncs04}
\bibliography{biblio}

\end{document}